\documentclass[12pt]{article}

\usepackage{amsthm,amssymb}
\usepackage{amsmath,amsfonts,latexsym}
\usepackage{float}
\restylefloat{table}
\usepackage{xcolor}
\usepackage{geometry} 
\usepackage[latin1]{inputenc}
\usepackage[english]{babel}
\usepackage{url}

\newtheorem{thm}{Theorem}[section]

\newtheorem{cor}[thm]{Corollary}
\theoremstyle{definition}
\newtheorem{rem}[thm]{Remark}
\newtheorem{defi}[thm]{Definition}

\newenvironment{tight_enumerate}{
\begin{enumerate}
  \setlength{\itemsep}{0pt}
  \setlength{\parskip}{0pt}
}{\end{enumerate}}

\newenvironment{tight_itemize}{
\begin{itemize}
  \setlength{\itemsep}{0pt}
  \setlength{\parskip}{0pt}
}{\end{itemize}}
\begin{document}


\vspace*{0.5cm}


\begin{center}
{\Large\bf Construction of directed strongly regular graphs via their orbit matrices and genetic algorithm}
\end{center}



\vspace*{0.5cm}

\begin{center}
    Dean Crnkovi\'c  \\
		(deanc@math.uniri.hr, ORCID: 0000-0002-3299-7859)\\				
	Andrea \v{S}vob\\
 		(asvob@math.uniri.hr, ORCID: 0000-0001-6558-5167\\
		and\\
		Tin Zrinski \\ 
		(tin.zrinski@math.uniri.hr, ORCID: 0000-0003-3620-2176)\\
		\bigskip
{\textit{Faculty of Mathematics, University of Rijeka}}\\
{\textit{Radmile Matej\v ci\'c 2, 51000 Rijeka, Croatia}}\\

\end{center}

\vspace*{0.5cm}

\begin{abstract}
In this paper, we introduce orbit matrices of directed strongly regular graphs ($DSRG$s). Further, we propose a method of constructing directed strongly regular graphs with prescribed automorphism group using genetic algorithm. In the construction, we use orbit matrices, i.e. quotient matrices related to equitable partitions of adjacency matrices of putative directed strongly regular graphs induced by an action of a prescribed automorphism group. Further, we apply this method to construct directed strongly regular graphs with parameters $(36,10,5,2,3)$, $(52,12,3,2,3)$, $(52,15,6,5,6)$, $(55,20,8,6,8)$ and $(55,24,12,11,10)$.
\end{abstract}

\bigskip

\noindent {\bf AMS classification numbers:} 05C20, 05E18, 05E30.


\vspace*{0.7 cm}

\noindent {\bf Acknowledgement} \\
This work has been fully supported by {\rm C}roatian Science Foundation under the projects 4571 and 5713. 

\section{Introduction} \label{intro}

In this paper, we are interested in constructing directed strongly regular graphs. We introduce orbit matrices of directed strongly regular graphs and describe a method that uses a genetic algorithm to construct directed strongly regular graphs 
with a prescribed automorphism group. In the construction, we use orbit matrices with respect to the prescribed automorphism group. This method is applied for construction of directed strongly regular graphs with parameters $(36,10,5,2,3)$, $(52,12, 3, 2, 3)$, $(52, 15, 6, 5, 6)$, $(55, 20, 8, 6, 8)$ and $(55, 24, 12, 11, 10)$. 

The paper is organized as follows. Section \ref{prelim} provides the necessary definitions and notation used throughout. In Section \ref{orb-mat}, we introduce a definition of orbit matrices of directed strongly regular graphs and give its basic properties, and in Section \ref{gen-alg} we give a brief overview of genetic algorithms. In Section \ref{construction}, we describe the algorithm for constructing directed strongly regular graphs that is used in this paper, and in Section \ref{newDSRG} we apply this algorithm to construct directed strongly regular graphs with parameters $(36,10,5,2,3)$, $(52,12, 3, 2, 3)$, $(52, 15, 6, 5, 6)$, $(55, 20, 8, 6, 8)$ and $(55, 24, 12, 11, 10)$. 

For the computations in this paper we used programs written in GAP \cite{gap} and Magma \cite{magma}.

\section{Preliminaries} \label{prelim}

A \textit{directed graph} or a \textit{digraph} is an ordered pair $\Gamma=(\mathcal{V},\mathcal{E})$, where $\mathcal{V}$ is a non-empty finite set and $\mathcal{E} \subseteq \{ (x,y) |\ x,y \in V \}$. 

The elements of $\mathcal{V}$ are called the \textit{vertices} of the graph $\Gamma$, and the elements of $\mathcal{E}$ are called the \textit{arcs} of $\Gamma$. We say that vertices $x$ and $y$ are adjacent if there is an arc $(x,y) \in \mathcal{E}$, which we denote by $x \to y$. We say that $x$ is the \textit{source vertex} of the arc $a=(x,y)$, and $y$ is the \textit{target vertex} of $a$. 

The out-degree of a vertex $x\in \mathcal{V}$ is the number of arcs $x\to z$ and the in-degree of a vertex $x\in \mathcal{V}$ is the number of arcs $w\to x$. A directed graph $\mathcal{V}$ is called \textit{$k-$regular} if each vertex $x\in \mathcal{V}$ has the out-degree and the in-degree equal to $k$. A directed graph $\mathcal{V}$ is called \textit{simple} if there are no arcs $x\to x$ (loops), for any $x\in \mathcal{V}$, and no multiple arcs with same source and target vertices. 

The \textit{adjacency matrix} $A$ of a directed graph $\Gamma$ on $v$ vertices $x_1, x_2, \ldots, x_v$ is a $v\times v$ matrix $A =[a_{ij}]$ such that $a_{ij}$ is the number of arcs $x_i \to x_j$.

Let $\Gamma_1=(\mathcal{V}_1,\mathcal{E}_1)$ and $\Gamma_2=(\mathcal{V}_2,\mathcal{E}_2)$ be two directed graphs. A bijective mapping $f:\mathcal{V}_1 \to \mathcal{V}_2$ is an \textit{isomorphism} from graph $\Gamma_1$ to graph $\Gamma_2$ if for every $x, y\in \mathcal{V}$ the following holds:
$$(x,y) \in \mathcal{E}_1 \text{ if and only if } (f(x), f(y)) \in \mathcal{E}_2.$$
An isomorphism of a directed graph $\Gamma=(\mathcal{V},\mathcal{E})$ to itself is an \textit{automorphism} of $\Gamma$. The set of all automorphisms of a directed graph $\Gamma$ forms a group called the \textit{full automorphism group} of ${\Gamma}$ and denoted by $Aut({\Gamma})$.

Duval \cite{duval} introduced the following definition. Let $\Gamma=(\mathcal{V},\mathcal{E})$ be a simple directed $k$-regular graph with $v$ vertices. We say that $\Gamma$ is a \textit{directed strongly regular graph} with parameters $(v,k,t,\lambda,\mu)$ if the number of directed paths of length 2 from a vertex $x$ to a vertex $y$ is 
\begin{itemize}
	\item $\lambda$ if there is an arc $x \to y$,
	\item $\mu$ if there is no arc $x \to y$, 
	\item $t$ if $x=y$.
\end{itemize} 
Directed strongly regular graphs with parameters $(v,k,t,\lambda,\mu)$ are denoted by $DSRG(v,k,t,\lambda,\mu)$.

Adjacency matrix $A$ of a directed strongly regular graph with parameters $(v,k,t,\lambda,\mu)$ satisfies: 
$$A^2+(\mu-\lambda)A-(t-\mu)I=\mu J,$$
$$AJ=JA=kJ,$$
where $I$ is the identity matrix of order $v$ and $J$ is the $v\times v$ matrix of all 1's.

It is obvious that a directed strongly regular graph with $t=k$ is a strongly regular graph.

A list of known directed strongly regular graphs with their parameters and a list of parameters that satisfy the necessary conditions for existence of a directed strongly regular graphs, together with information about constructions and non-existence of directed strongly regular graphs on $v$ vertices, $v\leq 110$, can be found in \cite{DSRGs}.

\section{Orbit matrices of directed strongly regular graphs} \label{orb-mat}

A partition $\Pi =\{ C_0,C_1,...,C_{t-1} \}$ of the vertices of a graph $\Gamma$ is called {\it equitable} (or regular) if for every pair of (not necessarily distinct) indices
$i,j \in \{0,1,...,t-1 \}$ there  is  a  nonnegative  integer $b_{ij}$ such  that  each  vertex $x \in C_i$ is adjacent with $b_{ij}$ vertices in $C_j$, regardless of the choice of $x$.  
The $t \times t$ matrix $B = [b_{ij}]$ is called a {\it quotient matrix} of $\Gamma$ with respect to the equitable partition $\Pi$.

Let a group $G$ act as an automorphism group of a directed strongly regular graph $\Gamma$ with parameters $(v,k,t,\lambda,\mu)$. 
Further, let $O_1,O_2, \ldots, O_b$ be the $G$-orbits on the set of vertices of $\Gamma$, and $|O_i|=n_i$ be the corresponding orbit lengths, for $i=1, \ldots,b$.

Let $A$ be the adjacency matrix of $\Gamma$. The orbits $O_1,O_2, \ldots, O_b$ divide the matrix $A$ into submatrices $A_{ij}$. This partition of the matrix $A$ is equitable, i.e. the row and column sum of $A_{ij}$ is constant, for $1 \le i,j \le b$. Denote by $r_{ij}$ the row sum of $A_{ij}$, and by $c_{ij}$ the column sum of $A_{ij}$. Then
$$r_{ij}n_i=c_{ij}n_j.$$    
Let us define $b \times b$ matrices $R=[r_{ij}]$ and $C=[c_{ij}]$. The matrix $R$ is called the row orbit matrix of $\Gamma$ with respect to the action of $G$, and $C$ is called the column orbit matrix of $\Gamma$ with respect to the action of $G$. Clearly, $R=C^{\top}$.

For a vertex $u$ denote by $N^+(u)$ the out-neighbourhood of $u$, i.e. the set of all vertices $v$ such that there is an arc $u \rightarrow v$ in $\Gamma$. Further, denote by $N^-(u)$ the in-neighbourhood of $u$, i.e. the set of all vertices $v$ such that there is an arc $v \rightarrow u$. 
Let $x$ be an element of the orbit $O_i$, and let $y_s$ be a representative of the orbit $O_s$, for $s = 1, \ldots, b$. Then the following holds.

\begin{align*}
\sum_{s=1}^b  r_{is} r_{sj}
&=\sum_{s=1}^b  | N^+(x) \cap O_s| \cdot | N^+(y_s) \cap O_j|
 =\sum_{s=1}^b | N^+(x) \cap O_s | \cdot \sum_{y \in O_j} | N^+(y_s) \cap \{ y \}| \\
&=\sum_{s=1}^b | N^+(x) \cap O_s | \cdot \sum_{y \in O_j} |\{ y_s \} \cap N^-(y)|
 =\sum_{s=1}^b \sum_{u \in N^+(x) \cap O_s} \sum_{y \in O_j} |\{ u \} \cap N^-(y)| \\
&=\sum_{s=1}^b \sum_{y \in O_j} \sum_{u \in N^+(x) \cap O_s} |\{ u \} \cap N^-(y)|
 =\sum_{s=1}^b \sum_{y \in O_j} | N^+(x) \cap O_s \cap N^-(y)| \\
&=\sum_{y \in O_j} \sum_{s=1}^b | N^+(x) \cap O_s \cap N^-(y)|
 =\sum_{y \in O_j} |N^+(x) \cap N^-(y)|.
\end{align*}

If $i \neq j$, then
\begin{equation} \label{eq-row-prod-1}
\sum_{y \in O_j} |N^+(x) \cap N^-(y)| = r_{ij} \lambda + (n_j - r_{ij}) \mu,
\end{equation}

and if $i = j$, then
\begin{equation} \label{eq-row-prod-2}
\sum_{y \in O_j} |N^+(x) \cap N^-(y)| = t + r_{ij} \lambda + (n_j - r_{ij} -1) \mu.
\end{equation}

That leads us to the following theorem.

\begin{thm} \label{thm-row-om}
Let $\Gamma$ be a directed strongly regular graph with parameters $(v,k,t,\lambda,\mu)$ and let $G$ be an automorphism group of $\Gamma$. Further, let $O_1,O_2, \ldots, O_b$ be the $G$-orbits on the set of vertices of $\Gamma$, and let $|O_i|=n_i$, $i=1, \ldots,b$, be the corresponding orbit lengths. If $R=[r_{ij}]$ is the row orbit matrix of $\Gamma$ with respect to the action of $G$, then the following hold
\begin{equation} \label{eq-orb-mat-1}
0 \le r_{ij} \le n_j, \ {\rm for} \ 1 \le i,j \le b,
\end{equation}
\begin{equation} \label{eq-orb-mat-2}
0 \le r_{ii} \le n_i -1, \ {\rm for} \ 1 \le i \le b,
\end{equation}
\begin{equation} \label{eq-orb-mat-3}
\sum_{j=1}^b r_{ij} = k, \ {\rm for} \ 1 \le i \le b,
\end{equation}
\begin{equation} \label{eq-orb-mat-4}
\sum_{s=1}^b  r_{is} r_{sj} = \delta_{ij}(t- \mu) + r_{ij} \lambda + (n_j - r_{ij}) \mu, \ {\rm for} \ 1 \le i,j \le b,
\end{equation}
where $\delta_{ij}$ is the Kornecker delta.
\end{thm}
\begin{proof}
The equations (\ref{eq-orb-mat-1}) and (\ref{eq-orb-mat-2}) follow directly from the definition of a row orbit matrix, and the equation (\ref{eq-orb-mat-3}) follows from the fact that every vertex of 
$\Gamma$ has the out-degree equal to $k$. The equation (\ref{eq-orb-mat-4}) follow from the equations (\ref{eq-row-prod-1}) and (\ref{eq-row-prod-2}).
\end{proof}

The following corollary follows from Theorem \ref{thm-row-om} and the equation $r_{ij}n_i=c_{ij}n_j$.

\begin{cor} \label{cor-col-om}
Let $\Gamma$ be a directed strongly regular graph with parameters $(v,k,t,\lambda,\mu)$ and let $G$ be an automorphism group of $\Gamma$. Further, let $O_1,O_2, \ldots, O_b$ be the $G$-orbits on the set of vertices of $\Gamma$, and let $|O_i|=n_i$, $i=1, \ldots,b$, be the corresponding orbit lengths. If $C=[c_{ij}]$ is the column orbit matrix of $\Gamma$ with respect to the action of $G$, then the following hold
\begin{equation} \label{eq-orb-mat-1}
0 \le c_{ij} \le n_i, \ {\rm for} \ 1 \le i,j \le b,
\end{equation}
\begin{equation} \label{eq-orb-mat-2}
0 \le c_{ii} \le n_i -1, \ {\rm for} \ 1 \le i \le b,
\end{equation}
\begin{equation} \label{eq-orb-mat-3}
\sum_{i=1}^b c_{ij} = k, \ {\rm for} \ 1 \le j \le b,
\end{equation}
\begin{equation} \label{eq-orb-mat-4}
\sum_{s=1}^b  c_{is} c_{sj} = \delta_{ij}(t- \mu) + c_{ij} \lambda + (n_i - c_{ij}) \mu, \ {\rm for} \ 1 \le i,j \le b,
\end{equation}
where $\delta_{ij}$ is the Kornecker delta.
\end{cor}

Based on Theorem \ref{thm-row-om} and Corollary \ref{cor-col-om}, we introduce the following definitions.

\begin{defi} \label{defi-row-om}
Let  $n_i$, $i=1, \ldots,b$, be positive integers such that $\sum_{i=1}^b n_i=v$. A $(b \times b)$-matrix $R = [r_{ij}]$, where $r_{ij}$, $1 \le i,j \le b$, are non-negative integers satisfying conditions
\begin{equation*} \label{eq-orb-mat-1}
0 \le r_{ij} \le n_j, \ {\rm for} \ 1 \le i,j \le b,
\end{equation*}
\begin{equation*} \label{eq-orb-mat-2}
0 \le r_{ii} \le n_i -1, \ {\rm for} \ 1 \le i \le b,
\end{equation*}
\begin{equation*} \label{eq-orb-mat-3}
\sum_{j=1}^b r_{ij} = k, \ {\rm for} \ 1 \le i \le b,
\end{equation*}
\begin{equation*} \label{eq-orb-mat-3}
\sum_{i=1}^b \frac{n_i}{n_j}r_{ij} = k, \ {\rm for} \ 1 \le j \le b,
\end{equation*}
\begin{equation*} \label{eq-orb-mat-4}
\sum_{s=1}^b  r_{is} r_{sj} = \delta_{ij}(t- \mu) + r_{ij} \lambda + (n_j - r_{ij}) \mu, \ {\rm for} \ 1 \le i,j \le b,
\end{equation*}
where $\delta_{ij}$ is the Kornecker delta, is called a row orbit matrix for a directed strongly regular graph with parameters $(v,k,t,\lambda,\mu)$ 
and the orbit lengths distribution $(n_1, n_2, \ldots n_b)$.
\end{defi}

\begin{defi} \label{defi-col-om}
Let  $n_i$, $i=1, \ldots,b$, be positive integers such that $\sum_{i=1}^b n_i=v$. A $(b \times b)$-matrix $C = [c_{ij}]$, where $c_{ij}$, $1 \le i,j \le b$, are non-negative integers satisfying conditions
\begin{equation*} \label{eq-orb-mat-1}
0 \le c_{ij} \le n_i, \ {\rm for} \ 1 \le i,j \le b,
\end{equation*}
\begin{equation*} \label{eq-orb-mat-2}
0 \le c_{ii} \le n_i -1, \ {\rm for} \ 1 \le i \le b,
\end{equation*}
\begin{equation*} \label{eq-orb-mat-3}
\sum_{i=1}^b c_{ij} = k, \ {\rm for} \ 1 \le j \le b,
\end{equation*}
\begin{equation*} \label{eq-orb-mat-3}
\sum_{j=1}^b \frac{n_j}{n_i} c_{ij} = k, \ {\rm for} \ 1 \le j \le b,
\end{equation*}
\begin{equation*} \label{eq-orb-mat-4}
\sum_{s=1}^b  c_{is} c_{sj} = \delta_{ij}(t- \mu) + c_{ij} \lambda + (n_i - c_{ij}) \mu, \ {\rm for} \ 1 \le i,j \le b,
\end{equation*}
where $\delta_{ij}$ is the Kornecker delta, is called a column orbit matrix for a directed strongly regular graph with parameters $(v,k,t,\lambda,\mu)$ 
and the orbit lengths distribution $(n_1, n_2, \ldots n_b)$.
\end{defi}

\begin{rem} \label{rem-orb-mat}
Orbit matrices from Definitions \ref{defi-row-om} and \ref{defi-col-om} may or may not correspond to a directed strongly regular graph with parameters $(v,k,t,\lambda,\mu)$. Those matrices can be used for a construction of directed strongly regular graphs with a presumed automorphism group in a similar way as orbit matrices of block designs are used for a construction of block designs 
(see \cite{dc-mop, janko-gaeta, janko-tvt}) and orbit matrices of strongly regular graphs are used for a construction of strongly regular graphs (see \cite{lam, dc-mm, dc-as}). In this paper, we use orbit matrix of directed strongly regular graphs, together with a genetic algorithm, to construct directed strongly regular graphs with a presumed automorphism group.
\end{rem}

\section{Genetic algorithm} \label{gen-alg}

Genetic algorithms (GA) are search and optimization heuristic  population based  methods which are inspired by the natural evolution process. In each step of the algorithm, a subset of the whole solution space, called \textit{population}, is being treated. The population consists of \textit{individuals}, and each individual has \textit{genes} that can be mutated and altered. 
Instead of finding an optimal solution within the whole solution space, the algorithm concentrates on optimizing the selected population. Every individual represents a possible solution (optimum), which is evaluated using the \textit{fitness function}. In each iteration of the algorithm, a certain number of best-ranked individuals - \textit{parents} is selected to create new better individuals - \textit{children}. Children are created by a certain type of recombination - \textit{crossover} and they replace the worst-ranked individuals in the population, in order to increase the chances for convergence to the local optimum. After children are obtained, a \textit{mutation} operator is allowed to occur (for the purpose to escape from a local optimum) and the next generation of the population is created. The process is repeated until a termination condition is reached. Common terminating conditions are: an individual is found that satisfies optimum criteria, stagnation takes place in the sense that successive iterations no longer produce better results, or a predefined maximal number of generations is reached. This method has been shown to be very efficient for solving a variety of optimization problems, including some NP-hard problems (see \cite{CorMat}), as well as problems where any feasible solution is optimal - as it is in the case of construction of directed strongly regular graphs. 

A genetic algorithm has been used in \cite{dc-tz} and \cite{bibd-gen} to construct block designs and in \cite{dc-ddd} to find unitals as substructures of symmetric designs.
In this paper, we will describe a method that uses a genetic algorithm for a construction of directed strongly regular graphs  
with a prescribed automorphism group using orbit matrices. This method is applied to construct 
directed strongly regular graphs with parameters $(36,10,5,2,3)$, $(52,12, 3, 2, 3)$, $(52, 15, 6, 5, 6)$, $(55, 20, 8, 6, 8)$ and $(55, 24, 12, 11, 10)$. To the best of our knowledge, this is the first time genetic algorithms are used for constructing directed strongly regular graphs.

\section{Combining orbit matrices and genetic algorithm to construct directed strongly regular graphs} \label{construction} 

Let $(v,k,t,\lambda,\mu)$ be admissible parameters for a directed strongly regular graph. Further, let $G$ be a finite group, let $(n_1,n_2,\ldots,n_b)$ be the vertex orbit lengths distribution under the action of $G$ for a directed strongly regular graph $(v,k,t,\lambda,\mu)$, and let $M$ be a row orbit matrix for parameters $(v,k,t,\lambda,\mu)$ and orbit lengths distribution $(n_1,n_2,\ldots,n_b)$. Our goal is to construct  adjacency matrices of directed strongly regular graphs with parameters $(v,k,t,\lambda,\mu)$ which admit the action of an automorphism group isomorphic to $G$ with the orbit lengths distribution $(n_1,n_2,\ldots,n_b)$, such that this action produces the orbit matrix $M$. In this paper we prescribe actions of prime order groups, i.e. groups isomorphic to $Z_p, p\in \mathbb{P}$. While indexing the orbit matrix $M$, we will use a genetic algorithm, which means that our search will not be exhaustive. In other words, there is a possibility that this method will not produce all directed strongly regular graphs, up to isomorphism, which can be constructed from the given orbit matrix $M$. 
While testing the algorithm on examples of known directed strongly regular graphs for various parameters $(v,k,t,\lambda,\mu)$, we have determined optimal values of parameters of a genetic algorithm for construction of directed  strongly regular graphs.

In this construction, the individuals are adjacency matrices of simple directed $k$-regular graphs on $v$ vertices which allow the action of a group $G$ considering the row orbit matrix $M$. In other words, the individuals are $(0,1)$-matrices of dimensions $v\times v$ with zeroes on the diagonal, whose row sums and column sums are $k$ and which allow the action of the group $G$ with the orbit lengths distribution $(n_1,n_2,\ldots,n_b)$, so that the action produce the orbit matrix  $M$. Our aim is to take an initial popluation that consists of a certain number of such randomly generated individuals  and, by using a genetic algorithm, to construct an individual that represents an adjacency matrix of a directed strongly regular graph $DSRG(v,k,t,\lambda,\mu)$. 

A gene of such an individual is the union of rows 
of the adjacency matrix which correspond to one vertex orbit. 

A bit is a submatrix which represents the intersection of rows and columns of the adjacency matrix which correspond to two vertex orbits. That means that bits correspond to the elements of the orbit matrix. 
Every bit is determined by its first row, whereas the other rows are uniquely determined by the action of the group $G$ on the first row.

For some elements of the orbit matrix, the corresponding bits are uniquely determined (we will call them fixed bits), while from some other elements of the orbit matrix there are more possibilities to construct the corresponding bit (we will call them non-fixed bits). Similarly, we have fixed and non-fixed genes.

The fitness function is defined as follows. For every two distinct vertices $v_i$ and $v_j$, we set $x_{ij}$ to be the number of directed paths of length 2 from $v_i$ to $v_j$ i.e., the dot product of the $i^{th}$ row and $j^{th}$ column of the adjacency matrix.

The fitness function is
\[ \sum_{v_i,v_j \in {\mathcal V}} 
		\begin{cases} 
						min \{x_{ij}, t \}, 			 &\text{if }v_i=v_j,\\
						min \{x_{ij}, \lambda \},       &\text{if there is an arc from }v_i\text{ to }v_j,\\ 
						min \{x_{ij}, \mu \}, 		 &\text{if there is no arc from }v_i\text{ to }v_j,\\
		\end{cases} \]
where ${\mathcal V}$ is the set of vertices of a graph. 

For such a fitness function, an individual will be an adjacency matrix of a directed strongly regular graph with parameters $(v,k,t,\lambda,\mu)$ if and only if the value of its fitness is  
$$vt+vk\lambda+v(v-k-1)\mu.$$

This is because in a directed strongly regular graph with parameters $(v,k,t,\lambda,\mu)$ each of  $v$ vertices has $t$ directed paths of length 2 from itself to itself; each of $v$ vertices has arcs from itself to $k$ vertices and for each such arc there are $\lambda$ directed paths of length 2; finally each of $v$ vertices has no arcs from itself to $(v-k-1)$ vertices and for each such case there are $\mu$ directed paths of length 2.

With the fitness function defined in such a way, the problem of finding an optimal solution, that is, an adjacency matrix of a directed strongly regular graph, is a maximization problem. The fitness values of individuals over the generations should increase up to the maximal value $vt+vk\lambda+v(v-k-1)\mu.$ When matrices that attain the necessary fitness to be adjacency matrices of a directed strongly regular graph, an isomorphism check is conducted so that only mutually non-isomorphic directed strongly regular graphs remain.

The crossover is defined in a way that the genes at some positions of the first parent are replaced with the genes at the same positions of the second parent, and vice versa.

The mutation is performed in a way that one or more bits of an individual are replaced with new, randomly generated bits. That means that we randomly permute the first row in a bit and its other rows are determined by the action of the prescribed automorphism group. 

Sometimes a population gets stuck in a local optimum, causing a stagnation. In order to escape from a local optimum, we reset the algorithm. In our algorithm we have two kinds of resets, complete and partial, which is explained in more details in Section \ref{explanation}.

\subsection{Pseudocode of the algorithm} \label{alg}

Here we present the basic pseudocode of our algorithm that has proved to be successful in obtaining directed strongly regular graphs from orbit matrices with a certain prescribed automorphism group. 
The pseudocode of the algorithm will be explained in Section \ref{explanation}.

\bigskip

\noindent \textbf{Main function GA}
\begin{tight_enumerate}
	\item NrOfCompleteResets $\leftarrow$ 0
	\item NrOfPartialResets $\leftarrow$ 0
	\item \textbf{while} NrOfCompleteResets $<$ MaxNrOfCompleteResets
	\item StartingPopulation $\leftarrow$ empty
	\item \quad \textbf{while} NrOfPartialResets $<$ MaxNrOfPartialResets 
	\item \quad \quad Population $\leftarrow$ GeneratePopulation(POP,StartingPopulation)
	\item \quad \quad f\_best $\leftarrow$ BestFitness(Population)
	\item \quad \quad f\_bestNrOfRepeats $\leftarrow$ 1
	\item \quad \quad \textbf{while} f\_best $<$ FitnessForDSRG \textbf{and} NrOfGenerations $<$ MaxNrOfGenerations
	\item \quad \quad \quad Shuffle(Population)
	\item \quad \quad \quad WorkOnPopulation(Population)
	\item \quad \quad \quad increase NrOfGenerations
	\item \quad \quad \quad f\_best $\leftarrow$ BestFitness(Population)
	\item \quad \quad \quad \textbf{if} f\_best did not increase \textbf{then}
	\item \quad \quad \quad \quad increase f\_bestNrOfRepeats
	\item \quad \quad \quad \textbf{end if}
	\item \quad \quad \quad \textbf{if} f\_bestNrOfRepeats $=$ f\_bestNrOfRepeatsMax \textbf{then}
	\item \quad \quad \quad \quad StartingPopulation $\leftarrow$ predefined percentage of best-ranked individuals
	\item \quad \quad \quad \quad conduct a partial reset of population
	\item \quad \quad \quad \textbf{end if}
	\item \quad \quad \textbf{end while}
	\item \quad \quad increase NrOfPartialResets
	\item \quad \textbf{end while}
	\item \quad increase NrOfCompleteResets
	\item \textbf{end while}
\end{tight_enumerate}

\noindent\textbf{Function WorkOnPopulation}
\begin{tight_enumerate}
	\item $i$ $\leftarrow$ 1
	\item \textbf{while} $i$ $<$ POP 
	\item \quad WorkOnFour($i$,Population)
	\item \quad $i$ $\leftarrow$ $i+4$
	\item \textbf{end while}
\end{tight_enumerate}

\noindent\textbf{Function WorkOnFour}
\begin{tight_enumerate}
	\item Parents $\leftarrow$ two best-ranked among $i$th and $(i+4)$th individual
	\item \textbf{while} Parent1 $=$ Parent2
	\item \quad MutatedParent2 $\leftarrow$ Mutation(Parent2)
	\item \quad \textbf{if} MutatedParent2 not in Population 
	\item \quad \quad Parent2 $\leftarrow$ MutatedParent2
	\item \quad \textbf{end if}
	\item \textbf{end while} 
	\item Children $\leftarrow$ Crossover(Parents)
	\item \textbf{repeat} 
	\item \quad MutatedChild1 $\leftarrow$ Mutation(Child1)
	\item \textbf{until} not MutatedChild1 in Population
	\item Child1 $\leftarrow$ MutatedChild1
	\item \textbf{repeat} 
	\item \quad MutatedChild2 $\leftarrow$ Mutation(Child2)
	\item \textbf{until} not MutatedChild2 in Population
	\item Child2 $\leftarrow$ MutatedChild2
	\item two worst-ranked among $i$th and $(i+4)$th individual $\leftarrow$ Children
\end{tight_enumerate}

While running experiments we noticed that directed strongly regular graphs with different parameters respond differently to this basic algorithm. For example, directed strongly regular graphs with some parameters are obtained more quickly if just one random gene position is crossed over between parents, whereas other directed strongly regular graphs respond better to crossover of multiple random genes. Similarly, the algorithm set to run for some parameters shows less stagnation in a local maximum and needs fewer partial or complete resets in order to obtain a directed strongly regular graph if fewer bits are mutated, whereas other parameters respond better to a mutation of more bits.

In general, for directed strongly regular graphs with a small number of vertices the algorithm performs better if multiple gene positions are crossed over between parents, while in the case of directed strongly regular graphs with a larger number of vertices the algorithm performs better if only one gene position is crossed over.

Further, a construction was more effective for the parameters $(v,k,t,\lambda,\mu)$ where the ratio $\frac{v}{k}$ is bigger, than when it is smaller. With respect to prescribed groups, the algorithm performs better for the groups of smaller order. The best results were obtained for the groups of order two and three.

Performing a series of experiments to optimize the parameters of the genetic algorithm we concluded that optimal values are: 

\begin{tight_itemize}
	\item population size POP$\ =100$,
	\item MaxNrOfGenerations$\ =100\,000$,
	\item mutation probability $p_m=100\%$,
	\item crossover probability $p_c=100\%$,
	\item number of genes which participate in a crossover NrGenesForCrossover$\ =1$,
	\item number of bits to be mutated NrBitsForMutation$\ =1$,
	\item FitnessForDSRGNrOfRepeatsMax$\ =100$,
	\item MaxNrOfPartialResets$\ =10$,
	\item MaxNrOfCompleteResets$\ =100$,
	\item StartingPercentage$\ =10\%$.
\end{tight_itemize}

\subsection{Explanation of the pseudocode of the algorithm} \label{explanation}

The main function used in this algorithm is the GA function. The algorithm can run up to a number of predefined complete resets (MaxNrOfCompleteResets) and partial resets are possible up to some number of predefined partial resets (MaxNrOfPartialResets). At the start of the algorithm, a population of a certain number of individuals (POP) is created as a random population, satisfying constraints of the orbit matrix. In each iteration of the algorithm, we work on this population in order to create individuals with better properties. The algorithm is set to run until an individual that meets our criteria (i.e. is an adjacency matrix of a directed strongly regular graph) is obtained, until stagnation takes place (in that case we make a partial reset of the population) or until the limit for a number of complete resets (MaxNrOfCompleteResets) is reached.  

An individual's fitness is assessed by the FitnessOfIndividual function, and an individual represents an adjacency matrix of a directed strongly regular graph if its fitness is FitnessForDSRG. 

A partial reset of the population means that a predefined percentage of the best-ranked individuals in the population (StartingPopulation) are kept in that population along with new randomly generated individuals. A partial reset of the population is conducted in order to escape from stagnation, which means that a few generations in a row attain the same local optimum (in our case a local maximum considering the way our fitness function is defined). In order to detect this, we define the best fitness (f\_best) in one generation to be the maximum of fitnesses of all individuals in that generation. Then local maximum is detected by the algorithm if a certain predefined number of generations (f\_bestNrOfRepeatsMax) stagnates at some best fitness (f\_best). A partial reset can also be conducted if a predefined maximal number of generations MaxNrOfGenerations is reached. If no solution is obtained after this predefined number of partial resets, then a complete reset is conducted. 

Individuals are created as block matrices based on the orbit matrix and a prescribed automorphism group. These matrices are created semi-randomly in the following sense: each entry of the orbit matrix is expanded respecting the vertex length distributions and this expansion is either unique (i.e. that parts of orbit matrices are fixed) or it can be expanded in more than one way. An adequate expansion of all entries of the orbit matrix may produce the adjacency matrix of a directed strongly regular graph. Considering the many possible candidates for an adjacency matrix of a directed strongly regular graph among these expansions, in cases where exhaustive search cannot be conducted, heuristic algorithms are a good option. 

The NewIndividual function creates a block matrix of semi-randomly generated genes using a NewGene function (each gene corresponds to a row and column of the orbit matrix corresponding to a representative of a vertex orbit). These genes consist of bits (which correspond to entries in that row and column of the orbit matrix) which are generated using the NewBit function.

The NewBit function expands an entry of the orbit matrix into a matrix with the desired number of rows and columns (respecting the orbit lengths distributions). The first row of that matrix is constructed randomly and all other rows are determined by the action of a prescribed automorphism group. Entries in the fixed part of the orbit matrix, of course, produce a unique such matrix.

The GeneratePopulation function generates a new population of POP individuals, either from a blank population or with already existing better individuals after a partial reset.

The selection used in the algorithm is a 4-tournament selection (for that reason, POP must be a multiple of 4) where in each iteration of the algorithm population gets divided into groups of four individuals, among which two better (with respect to the fitness function) are selected to be parents in the crossover. These two parents produce children which are then mutated and those two individuals replace two worse individuals in the tournament. During this procedure, if parents chosen for crossover are equal, one of the parents gets mutated until they become different in order to maintain diversity in the population. Similarly, when children are mutated, it is done so that the resulting individual is different from any other individual in the current population.

After each iteration of the algorithm, the whole population is shuffled to ensure that the individuals chosen to be parents from one generation are not the same as the parents from the previous generation.

The algorithm uses two types of modifications on individuals: crossover and mutations. The Mutation function replaces one or more bits in a gene of an individual with another bit taking into consideration the expected fitness of that gene. 
The Crossover function takes two better individuals from a 4-tournament as parents and recombines their genes at a certain number of positions (NrOfPositionsForCrossover). Those positions are chosen randomly.

\section{Construction of directed strongly regular graphs with parameters $(36,10,5,2,3)$, $(52,12, 3, 2, 3)$, $(52, 15, 6, 5, 6)$, \\ $(55, 20, 8, 6, 8)$ and $(55, 24, 12, 11, 10)$} \label{newDSRG}

We took examples of directed strongly regular graphs constructed by L. K. J\o rgensen \cite{Jorgensen} and examples that we constructed using the method given in \cite{dc-vmc-as}.
From those graphs we obtained all their orbit matrices for presumed action of groups $Z_2$ and $Z_3$. Taking those orbit matrices as input for our genetic algorithm, we obtained many more non-isomorphic directed strongly regular graphs with the same parameters. We list the results below.

From a $DSRG(36,10,5,2,3)$ by L. K. J\o rgensen we obtain the following DSRGs. 
\begin{table}[H]
\begin{center}
\begin{tabular}{|c|c|c|c|}
\hline
	$|\text{Aut}(\Gamma)|$ & $\text{Aut}(\Gamma)$ structure & \#DSRGs  \\ 
\hline
\hline
	72 & $(S_3 \times S_3) : Z_2$ & 2 \\
	36 & $S_3 \times S_3$ & 6 \\
	18 & $Z_3 \times S_3$ & 16 \\
	18 & $E_9 : Z_2$ & 4 \\
	12 & $D_6$ & 58 \\
	9 & $E_9$ & 16 \\
	8 & $D_4$ & 7 \\
	6 & $S_3$ & 70 \\
	6 & $Z_6$ & 158 \\
	4 & $Z_4$ & 12 \\
	4 & $E_4$ & 109 \\
	3 & $Z_3$ & 191 \\
	2 & $Z_2$ & 180 \\
\hline
\hline
\end{tabular}
\caption{829 $DSRG(36,10,5,2,3)$}
\label{DSRG-36-10-5-2-3}
\end{center}
\end{table}

From two $DSRG(52,12,3,2,3)$ that we constructed from the group $PSL(3,3)$ using the method described in \cite{dc-vmc-as}, we obtain the following. 
\begin{table}[H]
\begin{center}
\begin{tabular}{|c|c|c|c|}
\hline
	$|\text{Aut}(\Gamma)|$ & $\text{Aut}(\Gamma)$ structure & \#$DSRG$s  \\ 
\hline
\hline
	5616 & $PSL(3,3)$ & 2 \\
	4 & $Z_4$ & 1 \\
	4 & $E_4$ & 11 \\
	3 & $Z_3$ & 41 \\
	2 & $Z_2$ & 336 \\
\hline
\hline
\end{tabular}
\caption{391 $DSRG(52,12,3,2,3)$}
\label{DSRG-52-12-3-2-3}
\end{center}
\end{table}

From a $DSRG(52,15,6,5,4)$ constructed from $PSL(3,3)$ using the method from \cite{dc-vmc-as}, we obtain the following results.
\begin{table}[H]
\begin{center}
\begin{tabular}{|c|c|c|c|}
\hline
	$|\text{Aut}(\Gamma)|$ & $\text{Aut}(\Gamma)$ structure & \#$DSRG$s  \\ 
\hline
\hline
	5616 & $PSL(3,3)$ & 1 \\
	3 & $Z_3$ & 164 \\
\hline
\hline
\end{tabular}
\caption{165 $DSRG(52,15,6,5,4)$}
\label{DSRG-52-15-6-5-4}
\end{center}
\end{table}

From two $DSRG(55,20,8,6,8)$ constructed from $PSL(2,11)$ with the method from \cite{dc-vmc-as}, we obtain the following DSRGs. 
\begin{table}[H]
\begin{center}
\begin{tabular}{|c|c|c|c|}
\hline
	$|\text{Aut}(\Gamma)|$ & $\text{Aut}(\Gamma)$ structure & \#$DSRG$s  \\ 
\hline
\hline
	660 & $PSL(2,11)$ & 2 \\
	8 & $D_4$ & 7 \\
	2 & $Z_2$ & 56 \\
\hline
\hline
\end{tabular}
\caption{65 $DSRG(55,20,8,6,8)$}
\label{DSRG-55-20-8-6-8}
\end{center}
\end{table}

From a $DSRG(55,24,12,11,10)$ constructed by the method from \cite{dc-vmc-as}, using the group $PSL(2,11)$, we got the following DSRGs.
\begin{table}[H]
\begin{center}
\begin{tabular}{|c|c|c|c|}
\hline
	$|\text{Aut}(\Gamma)|$ & $\text{Aut}(\Gamma)$ structure & \#$DSRG$s  \\ 
\hline
\hline
	660 & $PSL(2,11)$ & 1 \\
	3 & $Z_3$ & 20 \\
      2 & $Z_2$ & 1\\
\hline
\hline
\end{tabular}
\caption{22 $DSRG(55,24,12,11,10)$}
\label{DSRG-55-24-12-11-10}
\end{center}
\end{table}


\noindent The directed strongly regular graphs constructed are available at: 

\url{https://github.com/TinZrinski/structures/blob/main/DSRG-36-10-5-2-3.txt},

\url{https://github.com/TinZrinski/structures/blob/main/DSRG-52-12-3-2-3.txt},

\url{https://github.com/TinZrinski/structures/blob/main/DSRG-52-15-6-5-4.txt},

\url{https://github.com/TinZrinski/structures/blob/main/DSRG-55-20-8-6-8.txt},

\url{https://github.com/TinZrinski/structures/blob/main/DSRG-55-24-12-11-10.txt}.









\end{document}